\documentclass{ws-idaqp}
\usepackage{graphicx}
\usepackage[super]{cite}

\newcommand{\E}{{\mathbb{E}}}
\newcommand{\R}{{\mathbb{R}}}

\begin{document}
	
	\markboth{Qingsong Wang, A. A. Dorogovtsev}{Small deviations for integrators in terms of It\^{o}--Wiener expansions}
	
	
	\title{Small deviations for integrators in terms of It\^{o}--Wiener expansions}

\author{Qingsong Wang}
	
	\address{Jilin Police College, Changchun, 130012, P.~R. China.\\
1115365358@qq.com}

\author{A.~A. Dorogovtsev}

\address{Institute of Mathematics, National Academy of Sciences of Ukraine, Ukraine.\\
andrey.dorogovtsev@gmail.com}

\maketitle

\begin{abstract}
In this paper, we investigate small deviations for a class of Gaussian processes, known as integrators. In general, such processes are non-Markovian and may even fail to be semimartingales. First, we investigate small deviations for the multi-dimensional Wiener process, focusing on the asymptotic behavior of the Fourier--Wiener transform of the associated hitting event. We analyze the Fourier--Wiener transform and the separate summands in its It\^{o}--Wiener expansion. Then, we establish the asymptotic limits of the associated multiple integrals with heat kernels. Finally, we obtain the small deviation asymptotic result for the integrator $\eta$ associated with the continuous linear operator.

\end{abstract}


\ccode{AMS Subject Classification: }

\section{Introduction}

The investigation of probability $\mathbb{P}(\|X\|<\varepsilon)$ as $\varepsilon\rightarrow0$, where $X$ is a stochastic process with a suitable norm $\|\cdot\|$, originated with the work of K. L. Chung \cite{Chu48} in 1948,  who established the first exact logarithmic asymptotics for Brownian motion under the uniform norm. In 1990s, M. A. Lifshits\cite{Lif98}, W. Linde\cite{Lin98} and W. V. Li\cite{Li99}, established some well-known asymptotic results not only for Brownian motion but also for some kind of Gaussian processes. In this paper, we study the small deviations for integrators. The integrator was first introduced by A. A. Dorogovtsev\cite{Dor06} for the purposes related to the theory of stochastic integration and anticipating stochastic differential equations. Integrators includes Brownian motion, Brownian bridge, Ornstein-Uhlenbeck process and other processes. We start from the small deviations for the multi-dimensional Wiener process, approaching it from the perspective of the asymptotic behavior of the Fourier--Wiener transform of the associated hitting event. Specifically, we analyze the behavior both of the Fourier--Wiener transform itself and of the individual summands in its series expansion. To this end, we first recall the parabolic boundary value problem satisfied by the Fourier--Wiener transform, and construct the transform via the Duhamel series representation associated with this problem. Furthermore, we establish the corresponding asymptotic limits for the associated multiple integrals with heat kernels. Finally, we obtained the small deviation asymptotic result for the integrator $\eta$ associated with the continuous linear operator $A$. The remainder of this paper is organized as follows. Section 2 studies the small deviations for the multi-dimensional Wiener process through the Fourier--Wiener transform, and Section 3 extends the analysis to integrators generated by continuous linear operators.

\section{Parabolic problems for the Fourier--Wiener transform of hitting events}
\subsection{Integrator's definition and examples}

\begin{definition}\cite{Dor06}
A centered Gaussian process $\{\eta(t);t\in[0,1]\}$ is called an integrator if there exists a positive constant $C>0$ such that for arbitrary $n\geq1$, partition $0=t_{0}<\cdots<t_{n}=1$ and real numbers $a_{0},\cdots,a_{n-1}$, the following relation holds,
\begin{equation*}
	\begin{aligned}
		\E\left(\sum_{k=0}^{n-1}a_{k}(\eta(t_{k+1})-\eta(t_{k}))\right)^{2}\leq C\sum_{k=0}^{n-1}a_{k}^{2}\Delta t_{k}^{2},
	\end{aligned}
\end{equation*}
where $\Delta t_{k}:=t_{k+1}-t_{k}$.
\end{definition}

This definition shows that for every function $f\in L^{2}([0,1])$, the integral
\begin{equation*}
\begin{aligned}
	\int_{0}^{1}f(t)d\eta(t)
\end{aligned}
\end{equation*}
is always well-defined, in the sense of the limit in the square mean of integrals. Furthermore, the following lemma gives a representation of the integrators in terms of Gaussian white noise on a Hilbert space.

\begin{lemma}\cite{Dor06}
	Let $\{\eta(t);t\in[0,1]\}$ be an integrator. Then there exists a Wiener process $\{w(t);t\in[0,1]\}$ and a continuous linear operator $A$ on the space $L^{2}([0,1])$ into itself such that
	\begin{equation*}
		\begin{aligned}
			\eta(t)=\int_{0}^{1}A(\mathbf{1}_{[0,t]})(s)dw(s).
		\end{aligned}
	\end{equation*}
\end{lemma}

Using the previous representation, we will give two examples of integrators.

\begin{example}
	Consider a simple case, where $A=I$, the integrator $\eta(t)$ is exactly the Wiener process.
\end{example}

\begin{example}
	Consider the Brownian bridge $B(t)$ defined as follows
	\begin{equation*}
		\begin{aligned}
			B(t)=w(t)-tw(1).
		\end{aligned}
	\end{equation*}
	We can divide $B(t)$ into two parts: the first part $w(t)$ is precisely the standard Brownian motion introduced in our previous example, while the second term admits the representation
	\begin{equation*}
		\begin{aligned}
			tw(1)=\int_{0}^{1}\int_{0}^{1}\mathbf{1}_{[0;t]}(y)dydw(s).
		\end{aligned}
	\end{equation*}
	We may then express the operator $A$ as
	\begin{equation*}
		\begin{aligned}
			Af(y)=If(y)-\int_{0}^{1}f(y)dy,
		\end{aligned}
	\end{equation*}
	for any $f \in L^{2}([0,1])$, where $I$ denotes the identity operator on $L^2([0,1])$.
\end{example}

\subsection{The Fourier--Wiener transform associated with It\^{o}--Wiener expansion} 
To define the Fourier--Wiener transform, we start with a d-dimensional standard Wiener process $\{w(t);t\in[0,1]\}$. For any $h\in L^2([0,1],\mathbb{R}^d)$, we introduce the stochastic exponential
\begin{equation*}
	\varepsilon(h)=\exp\left\{\int_{0}^{1}\langle h(t),dw(t)\rangle-\frac{1}{2}\int_{0}^{1}\|h(t)\|^2_{\mathbb{R}^d},dt\right\},
\end{equation*}
where $\langle\cdot,\cdot\rangle_{\mathbb{R}^d}$ and $\|\cdot\|_{\mathbb{R}^d}$ stand for the standard Euclidean inner product and norm on $\mathbb{R}^d$, respectively.
For a square-integrable Brownian functional $\alpha\in L^2(\Omega,\mathcal{F}_1,\mathbb{P})$, its Fourier--Wiener transform is defined as follows
\begin{equation*}
	\mathcal{T}(\alpha)(h)=\mathbb{E}\left[\alpha \varepsilon(h)\right].
\end{equation*}
If $\alpha$ admits the It\^{o}--Wiener expansion
\begin{equation*}\alpha= \mathbb{E}[\alpha] + \sum_{n=1}^{\infty}\int_{\Delta_n}a_{n}(t_1,\dots,t_n),dw(t_1)\cdots dw(t_n),
\end{equation*}
where $\Delta_n:=\{(t_1,\dots,t_n)\in[0,1]^n:0\leq t_1\leq\cdots\leq t_n\leq1\}$ is the n-dimensional standard simplex, and $a_n\in L^2(\Delta_n,(\mathbb{R}^d)^{\otimes n})$ is a deterministic symmetric kernel, then the Fourier--Wiener transform of $\alpha$ admits the representation
\begin{equation*}\mathcal{T}(\alpha)(h)=\mathbb{E}[\alpha] +\sum_{n=1}^{\infty}\int_{\Delta_n}\langle a_{n}(t_1,\dots,t_n),h(t_1)\otimes\cdots\otimes h(t_n)\rangle_n,dt_1\cdots dt_n.
\end{equation*}
Here $\otimes$ denotes the tensor product on $\mathbb{R}^d$, and $\langle\cdot,\cdot\rangle_n$ is the canonical inner product on $(\mathbb{R}^d)^{\otimes n}$ induced by $\langle\cdot,\cdot\rangle_{\mathbb{R}^d}$, defined for simple tensors by $\langle u_1\otimes\cdots\otimes u_n,v_1\otimes\cdots\otimes v_n\rangle_n=\prod_{i=1}^n\langle u_i,v_i\rangle_{\mathbb{R}^d}$.

\subsection{Hitting functional for Wiener process and it's Fourier--Wiener transform}
Next, consider the Fourier--Wiener transform of the hitting event of a domain boundary. For $x\in \mathcal{D}$ (where $\mathcal{D} \subset \mathbb{R}^{d}$ is a domain with smooth boundary) and $s\in[0,1]$, define the indicator random variable
\begin{equation*}
	W_{x,s} := 
	\begin{cases}
		1, & \text{if there exists } \tau\in[s,1] \text{ such that } x+w(\tau)-w(s)\in\partial \mathcal{D}, \\
		0, & \text{otherwise}.
	\end{cases}
\end{equation*}
The expectation $\mathbb{E}[W_{x,s}]$ is the probability that the Wiener process starting from $x$ at time $s$ hits the boundary $\partial \mathcal{D}$ by time $1$.

Then, for a fixed $h\in C^{1}([0,1],\R^{d})$, the function
\begin{equation*}
	\begin{aligned}
		U^{h}_{x,t}=\mathcal{T}(W_{x,t})(h)
	\end{aligned}
\end{equation*}
satisfies the boundary value problem \cite{Wan24}
\begin{equation*}
	\begin{aligned}
		\left\{
		\begin{array}{lcl}
			\frac{\partial}{\partial{t}}U^{h}_{x,t}=-\frac{1}{2}\triangle_{x}U^{h}_{x,t}-(h(t),\nabla_{x}U_{x,t}), \\
			U^{h}_{x,1}=0, \quad x\in \mathcal{D},\\
			U^{h}_{x,t}=1, \quad x\in \partial \mathcal{D}.\\
		\end{array}
		\right.
	\end{aligned}
\end{equation*}

The Duhamel series of $U^{h}_{x,t}$ is 
\begin{equation*}
	\begin{aligned}
		U^{h}_{x,t}=&U^{0}_{x,t}+\sum_{n=1}^{\infty}\int_{t\leq r_{1}\leq\cdots\leq r_{n}\leq1}\int_{\partial \mathcal{D}}\stackrel{n}{\cdots}\int_{\partial \mathcal{D}}G(t,r_{1},x,z_{1})\\
		&(\nabla_{z_{1}}G(r_{1},r_{2},z_{1},z_{2})\otimes\cdots\otimes\nabla_{z_{n}}U^{0}_{z_{n},r_{n}},h(r_{1})\otimes\cdots\otimes h(r_{n}))_{n}\\
		&dz_{1}\cdots dz_{n}dr_{1}\cdots dr_{n},
	\end{aligned}
\end{equation*}
where the zero-order term $U^{0}_{x,t}$ can be expressed as follows
\begin{equation*}
	\begin{aligned}
		U^{0}_{x,t}=\int_{t}^{1}\int_{\partial \mathcal{D}}G'(t,r,x,y)\delta(dy)dr,
	\end{aligned}
\end{equation*}
$G(s,t,x,y)$ denote the Green function \cite{Ric94} for the parabolic problem, $\otimes$ represents the tensor product, $\nabla_{z_{k}}$ denotes the gradient operator with respect to the variable $z_{k}$ ($k=1,2,\dots,n$).

\subsection{Asymptotic behavior of $\mathcal{T}(W^{\varepsilon}_{0,0})(h)$}
In this section, we study the asymptotic behavior of $\mathcal{T}(W^{\varepsilon}_{0,0})(h)$. To this end, we first recall that $\varepsilon(h)$ serves as a Radon--Nikodym density. Let $P$ denote the law of the standard Brownian motion $w(t)$, and let $Q$ denote the law of the drifted Brownian motion
\begin{equation*}
	\tilde{w}(t) = w(t) - \int_0^t h(s) \, ds.
\end{equation*}
Then $\varepsilon(h) = \frac{dQ}{dP}$, the Radon--Nikodym derivative of $Q$ with respect to $P$. In this setting, the asymptotic behavior of the hitting probability for the drifted Brownian motion coincides with that of the Fourier--Wiener transform of $W^{\varepsilon}_{0,0}$. We now state a preliminary result that will be used in the subsequent theorem.

\begin{lemma}\cite{Lif98}
	Let $\mu$ be a centered Gaussian measure on a separable Banach space $(E, \|\cdot\|)$, and let $H_{\mu}$ denote the reproducing kernel Hilbert space associated with $\mu$, with norm $\|\cdot\|_{H_{\mu}}$. Then, for any $h \in H_{\mu}$, the following limit holds
	\begin{equation*}
		\lim_{\varepsilon \to 0} \frac{\mu(B(h,\varepsilon))}{\mu(B(0,\varepsilon))} = \exp\left\{-\frac{1}{2} \|h\|_{H_\mu}^2\right\},
	\end{equation*}
	where $B(x,\varepsilon):=\{y \in E : \|y-x\| \leq \varepsilon\}$ denotes the closed ball of radius $\varepsilon$ centered at $x$.
\end{lemma}

Consider the process
\begin{equation*}
	M(t) := w(t) + \int_{0}^{t} h(s) \, ds, \quad t \in [0,1].
\end{equation*}
We define $W^{\varepsilon,h}_{x,t}$ accordingly. For $x \in \mathcal{D}$ and $s \in [0,1]$, introduce the stopping time
\begin{equation*}
	\tau_{x,s} := \inf\left\{ t \geq s : x + M(t) - M(s) \in \partial \mathcal{D} \right\},
\end{equation*}
and set
\begin{equation*}
	W^{\varepsilon,h}_{x,s} :=
	\begin{cases}
		1, & \tau_{x,s} \leq 1, \\
		0, & \tau_{x,s} > 1.
	\end{cases}
\end{equation*}

In the $d$-dimensional case, let $h=(h_{1}, \dots, h_{d}) \in L^2([0,1]; \R^{d})$, and let $w(t)= (w_{1}(t), \dots, w_{d}(t))$ be a $d$-dimensional standard Brownian motion, where $w_{1}(t), \dots, w_{d}(t)$ are independent one-dimensional standard Brownian motions. Combining this with the small deviation results for $d$-dimensional Brownian motion, we have the following result.

\begin{theorem}\label{tm2.2}
In the $d$-dimensional case, consider $W^{\varepsilon}_{x,s}$ on the ${B}(0,\varepsilon)\subset\R^{d}$, assume $h \in L^2([0,1])$, denote $|h|^2=\sum_{i=1}^{d} h_{i}^{2}$. Then
\begin{equation*}
	\begin{aligned}
		1-\mathcal{T}(W^{\varepsilon}_{0,0})(h) \sim \frac{j_{\nu,1}^{\nu-1}}{\Gamma(\nu+1)2^{\nu-1}J_{\nu+1}(j_{\nu,1})} e^{- \frac{1}{2} \int_0^1 |h(s)|^2 \, ds-j^{2}_{\nu,1}/2\varepsilon^{2}}, \quad \varepsilon\rightarrow 0,
	\end{aligned}
\end{equation*}
where $\Gamma( \cdot )$ is the Gamma function, $\nu=\frac{d}{2}-1$,
\begin{equation*}
	\begin{aligned}
		J_{\nu}(x)=\sum_{k=0}^{\infty}\frac{(-1)^{k}(\frac{x}{2})^{\nu+2k}}{k!\Gamma(\nu+k+1)},
	\end{aligned}
\end{equation*}
and $j_{\nu,1}$ is the first positive zero of $J_{\nu}(x)$.
\end{theorem}

\begin{proof}
The stochastic exponent $\varepsilon(h)$ is a Radon--Nikodym density and
\begin{equation*}
	\begin{aligned}
		\mathcal{T}(W^{\varepsilon}_{0,0})(h)=\E [W^{\varepsilon}_{0,0}\varepsilon(h)]=\E [W^{\varepsilon,h}_{0,0}].
	\end{aligned}
\end{equation*}
If $h'$ exists and is continuous, using integration by parts we obtain
\begin{equation*}
	\int_0^1 h(s) \, dw(s) = -\int_0^1 h'(s) w(s) \, ds + h(1) w(1).
\end{equation*}
This representation is continuous with respect to $w$ in the uniform norm. So, the hitting probability for the process $w(t)+\int_{0}^{t}h(s)ds$ can be expressed as follows
\begin{equation*}
	\begin{aligned}
		1-\mathcal{T}(W^{\varepsilon}_{0,0})(h)=&\int_{B(0,\varepsilon)}e^{-\int_0^1 h'(s)wds+h(1)w(1) - \frac{1}{2} \int_0^1 |h|^2(s) \, ds}\mu(dw)\\
		=&\frac{\int_{B(0,\varepsilon)}e^{-\int_0^1 h'(s)wds+h(1)w(1) - \frac{1}{2} \int_0^1 |h|^2(s) \, ds}\mu(dw)}{\mu(B(0,\varepsilon))}\mu(B(0,\varepsilon)),
	\end{aligned}
\end{equation*}
where $\mu$ is the Wiener measure. Since the exponential term is continuous with respect to $w$ in the uniform norm, we have
\begin{equation*}
	\begin{aligned}
	&\lim_{\varepsilon \to 0} \frac{\int_{B(0,\varepsilon)} \exp\left\{ -\int_0^1 h'(s) w(s) \, ds + h(1) w(1) - \frac{1}{2} \int_0^1 |h(s)|^2 \, ds \right\} \mu(dw)}{\mu(B(0,\varepsilon))}\\
	=& \exp\left\{ -\frac{1}{2} \int_0^1 |h(s)|^2 \, ds \right\}.
	\end{aligned}
\end{equation*}
Combining with the small deviation for Brownian motion, we can get the asymptotic result when $\varepsilon\rightarrow 0$
\begin{equation*}
	\begin{aligned}
		1-\mathcal{T}(W^{\varepsilon}_{0,0})(h)
		\sim \frac{j_{\nu,1}^{\nu-1}}{\Gamma(\nu+1)2^{\nu-1}J_{\nu+1}(j_{\nu,1})} e^{- \frac{1}{2} \int_0^1 |h(s)|^2 \, ds-j^{2}_{\nu,1}/2\varepsilon^{2}}.
	\end{aligned}
\end{equation*}
Consider the Cameron--Martin space
\begin{equation*}
	\begin{aligned}
		H_\mu = \left\{ f: f(0)=0,\ f \text{ is absolutely continuous},\ \int_0^1 |f'(s)|^2 ds < \infty \right\},
	\end{aligned}
\end{equation*}
equipped with the norm $\|f\|_{H_\mu} = \left( \int_0^1 |f'(s)|^2 ds \right)^{1/2}$. The drift function $\alpha(t) = \int_0^t h(s) ds$ belongs to $H_\mu$ if and only if $h \in L^2([0,1])$, and in this case $\|\alpha\|_{H_\mu}^2 = \int_0^1 |h|^2(s) ds$. By the asymptotic formula for the small ball probability \cite{Vla15} ratio of infinite-dimensional Gaussian measures, we obtain
\begin{equation*}
	\begin{aligned}
		\lim_{\varepsilon \to 0} \frac{\int_{B(0,\varepsilon)}{\varepsilon(h)}\mu(dw)}{\mu(B(0,\varepsilon))} = e^{-\frac{1}{2} \|\alpha\|_{H_\mu}^2} = e^{-\frac{1}{2} \int_0^1 |h|^2(s) \, ds}.
	\end{aligned}
\end{equation*}

\end{proof}

\subsection{Limits for separate summands of the Fourier--Wiener transform}
Based on the It\^{o}--Wiener expansion of $W^{\varepsilon}_{0,0}$, we may represent its Fourier--Wiener transform  via the series expansion
\begin{equation*}
	\begin{aligned}
		\mathcal{T}(W^{\varepsilon}_{0,0})(h) = \sum_{n=0}^{\infty} \mathcal{T}^{\varepsilon}_n(h).
	\end{aligned}
\end{equation*}
We can establish the following result.

\begin{theorem}\label{tm2.3}
	In the $d$-dimensional case, let $h \in L^2([0,1]; \mathbb{R}^d)$ and denote $\|h\|^2 = \int_0^1 |h(t)|^2 \, dt$. For every nonnegative integer $k$, we have
	\begin{equation*}
		\lim_{\varepsilon\to0} \mathcal{T}^{\varepsilon}_{2k}(h) \, e^{j_{\nu,1}^2/(2\varepsilon^2)}
		= \frac{j_{\nu,1}^{\nu-1}}{\Gamma(\nu+1)2^{\nu-1}J_{\nu+1}(j_{\nu,1})} \, \frac{(-1)^{k+1}}{2^k k!} \|h\|^{2k},
	\end{equation*}
	and
	\begin{equation*}
		\mathcal{T}^{\varepsilon}_{2k+1}(h) = 0.
	\end{equation*}
\end{theorem}

\begin{proof}
Denote
\begin{equation*}
	\begin{aligned}
		I_{n}=\int_{0}^{1} \cdots \int_{0}^{1}h(t_{1})\cdots h(t_{n})dw(t_{1}) \cdots dw(t_{n}).
	\end{aligned}
\end{equation*}
Consider the process
\begin{equation*}
	\begin{aligned}
		M_{1}= \exp \left( \int_{0}^{1}h(s)dw(s) - \frac{1}{2} \int_{0}^{1}{h(s)}^{2} ds \right),
	\end{aligned}
\end{equation*}
which has the It\^{o}--Wiener expansion
\begin{equation*}
	\begin{aligned}
		M_{1}= \sum_{n=0}^{\infty}\frac{I_{n}}{n!}.
	\end{aligned}
\end{equation*}
Then we have
\begin{equation*}
	M_1 = \exp\left( X - \frac{1}{2}\|h\|^2 \right), \quad \text{where } X = \int_0^1 h(s) \, dw(s).
\end{equation*}
Note that $X \sim N(0, \|h\|^2)$. Introduce the standardized variable $Z = X / \|h\|$, which is standard normal. The generating function for the Hermite polynomials is
\begin{equation*}
	\sum_{n=0}^{\infty} \frac{H_n(Z)}{n!} \|h\|^n = \exp\left( \|h\| Z - \frac{1}{2}\|h\|^2 \right) = \exp\left( X - \frac{1}{2}\|h\|^2 \right).
\end{equation*}
Comparing the two expansions of $M_1$, we obtain for every $n$,
\begin{equation}\label{Hermit1}
	\begin{aligned}
		I_n = \|h\|^n H_n(Z).
	\end{aligned}
\end{equation}
Denote $W^{\varepsilon}_{0,0} = 1_{\{\tau_\varepsilon \leq 1\}}$, where $\tau_\varepsilon = \inf\{t>0: |w(t)|=\varepsilon\}$. We aim to find the asymptotics of $\mathbb{E}[W^{\varepsilon}_{0,0} I_n]$ as $\varepsilon \to 0$. Since $\mathbb{E}[H_n(Z)] = 0$ for $n \geq 1$, we have
\begin{equation*}
	\begin{aligned}
		\mathbb{E}[1_{\{\tau_\varepsilon \leq 1\}} H_n(Z)] + \mathbb{E}[1_{\{\tau_\varepsilon > 1\}} H_n(Z)] = 0,
	\end{aligned}
\end{equation*}
so that
\begin{equation*}
	\begin{aligned}
		\mathbb{E}[W^{\varepsilon}_{0,0} I_n] = \|h\|^{n} \mathbb{E}[1_{\{\tau_\varepsilon \leq 1\}} H_n(Z)] = -\|h\|^{n} \mathbb{E}[1_{\{\tau_\varepsilon > 1\}} H_n(Z)].
	\end{aligned}
\end{equation*}
On the event $\{\tau_\varepsilon > 1\}$, the Brownian motion satisfies $|w(t)| < \varepsilon$ for all $t \in [0,1]$. Using integration by parts, we have
\begin{equation*}
	X = \int_0^1 h(t) \, dw(t) = h(1) w(1) - \int_0^1 w(t) h'(t) \, dt,
\end{equation*}
and it follows that there exists a constant $C$ (depending on $h$) such that $|X| \leq C \varepsilon$ uniformly on the event $\{\tau_\varepsilon > 1\}$. Hence $|Z| = |X|/{\|h\|} \leq (C/{\|h\|}) \varepsilon$, i.e., $Z = O(\varepsilon)$ on this event. Expand $H_n(Z)$ around $Z=0$. The Hermite polynomials satisfy $H_n(0) = 0$ when $n=2k+1$ for some integer $k\geq0$, and for $n=2k$ with $k\geq0$,
\begin{equation*}
	H_{2k}(0) = (-1)^k (2k-1)!!, \quad (2k-1)!! = 1 \cdot 3 \cdots (2k-1).
\end{equation*}
For $n=2k+1$, since $H_n(0)=0$ and $H_n$ is an odd function, we have $H_n(Z) = O(\varepsilon)$ uniformly on $\{\tau_\varepsilon > 1\}$, so that
\begin{equation*}
	\mathbb{E}\left[1_{\{\tau_\varepsilon > 1\}} H_n(Z)\right] = O\left(\varepsilon \mathbb{P}(\tau_\varepsilon > 1)\right) = o\left(\mathbb{P}(\tau_\varepsilon > 1)\right).
\end{equation*}
In fact, by symmetry, the expectation is exactly zero. For $n=2k$, since $H_{2k}$ is an even function and the next term in its Taylor expansion is of order $Z^2$, we have
\begin{equation*}
	H_{2k}(Z) = (-1)^k (2k-1)!! + O(\varepsilon^2)
\end{equation*}
uniformly on $\{\tau_\varepsilon > 1\}$. It then follows that
\begin{equation*}
	\mathbb{E}\left[1_{\{\tau_\varepsilon > 1\}} H_{2k}(Z)\right] = (-1)^k (2k-1)!! \, \mathbb{P}(\tau_\varepsilon > 1) + O(\varepsilon^2)\mathbb{P}(\tau_\varepsilon > 1).
\end{equation*}
Since $O(\varepsilon^2)\mathbb{P}(\tau_\varepsilon > 1)=o(\mathbb{P}(\tau_\varepsilon > 1))$, it follows that
\begin{equation*}
	\begin{aligned}
		\mathbb{E}[1_{\{\tau_\varepsilon > 1\}} H_{2k}(Z)] = (-1)^k (2k-1)!! \, \mathbb{P}(\tau_\varepsilon > 1) + o(\mathbb{P}(\tau_\varepsilon > 1)).
	\end{aligned}
\end{equation*}
Substituting into the expression for $\mathbb{E}[W^{\varepsilon}_{0,0} I_n]$ yields
\begin{equation*}
	\begin{aligned}
		\mathbb{E}[W^{\varepsilon}_{0,0} I_{2k}] =& -\|h\|^{2k} \mathbb{E}[1_{\{\tau_\varepsilon > 1\}} H_{2k}(Z)]\\
		=& (-1)^{k+1} (2k-1)!! \, \|h\|^{2k} \, \mathbb{P}(\tau_\varepsilon > 1) + o\big(\|h\|^{2k} \mathbb{P}(\tau_\varepsilon > 1)\big).
	\end{aligned}
\end{equation*}
Recall the small deviation asymptotics for Brownian motion
\begin{equation*}
	\begin{aligned}
		\mathbb{P}(\tau_\varepsilon > 1) \sim \frac{j_{\nu,1}^{\nu-1}}{\Gamma(\nu+1)2^{\nu-1}J_{\nu+1}(j_{\nu,1})}e^{-j^{2}_{\nu,1}/2\varepsilon^{2}}, \quad \varepsilon \to 0.
	\end{aligned}
\end{equation*}
Hence
\begin{equation*}
	\begin{aligned}
		\mathbb{E}[W^{\varepsilon}_{0,0} I_{2k}] \sim (-1)^{k+1} (2k-1)!! \, \|h\|^{2k} \cdot \frac{j_{\nu,1}^{\nu-1}}{\Gamma(\nu+1)2^{\nu-1}J_{\nu+1}(j_{\nu,1})}e^{-j^{2}_{\nu,1}/2\varepsilon^{2}} \quad \varepsilon \to 0.
	\end{aligned}
\end{equation*}
For $n=2k+1$, the result is exactly zero.
Thus, for general $\|h\|$, we obtain
\begin{equation*}
	\begin{aligned}
		\mathbb{E}[W^{\varepsilon}_{0,0} I_n] \sim
		\begin{cases}
			0, & n=2k+1, \\
			(-1)^{k+1} (2k-1)!! \, \|h\|^{2k} \, \frac{j_{\nu,1}^{\nu-1}}{\Gamma(\nu+1)2^{\nu-1}J_{\nu+1}(j_{\nu,1})}e^{-j^{2}_{\nu,1}/2\varepsilon^{2}}, & n = 2k,
		\end{cases}
		 \quad \varepsilon \to 0.
	\end{aligned}
\end{equation*}
By the expansion, one can easily get
\begin{equation*}
	\begin{aligned}
		\mathcal{T}^{\varepsilon}_n(h) = \frac{1}{n!}\mathbb{E}[W^{\varepsilon}_{0,0} I_n].
	\end{aligned}
\end{equation*}
By the relation
\begin{equation*}
	\begin{aligned}
		\frac{(2k-1)!!}{(2k)!}=\frac{1}{2^k k!}.
	\end{aligned}
\end{equation*}
Thus, for every nonnegative integer \(k\)
\begin{equation*}
	\begin{aligned}
		\lim_{\varepsilon\to0} \mathcal{T}^{\varepsilon}_{2k}(h)\, e^{j^{2}_{\nu,1}/2\varepsilon^{2}}
		= \frac{j_{\nu,1}^{\nu-1}}{\Gamma(\nu+1)2^{\nu-1}J_{\nu+1}(j_{\nu,1})}\, \frac{(-1)^{k+1}}{2^k k!} \|h\|^{2k},
	\end{aligned}
\end{equation*}
and
\begin{equation*}
	\begin{aligned}
		\mathcal{T}^{\varepsilon}_{2k+1}(h)= 0.
	\end{aligned}
\end{equation*}

\end{proof}

By the known result, in the one-dimensional case
\begin{equation}\label{eq11.11}
\begin{aligned}
	\mathcal{T}(W^{\varepsilon}_{0,0})(h)=&a^{\varepsilon}_{0,0}+\sum_{n=1}^{\infty}\int_{0\leq t_{1}\leq\cdots\leq t_{n}\leq1}a^{\varepsilon}_{n}(t_{1},\cdots,t_{n})h(t_{1})\cdots h(t_{n})dt_{1}\cdots dt_{n}\\
	=&a^{\varepsilon}_{0,0}+\sum_{n=1}^{\infty}\int_{0\leq t_{1}\leq\cdots\leq t_{n}\leq1}\int_{\partial B(0,\varepsilon)}\stackrel{n}{\cdots}\int_{\partial B(0,\varepsilon)}G_{\varepsilon}(0,t_{1},0,z_{1})\\
	&\nabla_{z_{1}}G_{\varepsilon}(t_{1},t_{2},z_{1},z_{2}) \cdots \nabla_{z_{n}}a^{\varepsilon}_{z_{n},t_{n}} h(t_{1})\cdots h(t_{n})dz_{1}\cdots dz_{n}dt_{1}\cdots dt_{n},
\end{aligned}
\end{equation}
where
\begin{equation*}
\begin{aligned}
	a^{\varepsilon}_{x,s}=\int_{0}^{1}\int_{\partial B(0,\varepsilon)}G_{\varepsilon}'(s,t,x,y)\delta(dy)dt.
\end{aligned}
\end{equation*}
This $a^{\varepsilon}_{0,0}$ is exactly the hitting probability $P\{\sup_{s\in[0,1]}\|w(s)\|=\varepsilon\}$. From this, we can also derive the asymptotic results for the multiple Green functions. Before presenting the next theorem, we first give a construction of the surface delta function, it satisfies
\begin{equation}\label{delta}
	\begin{aligned}
		k! \int_{0\leq s_1\leq\cdots\leq s_{2k}\leq1} \delta_k(t_1,\dots,t_{2k}) h(t_1)\cdots h(t_{2k}) dt_1\cdots dt_{2k} = \left( \int_0^1 h^2(t) dt \right)^k,
	\end{aligned}
\end{equation} 
where $\delta_k$ is defined by
\begin{equation*}
	\begin{aligned}
		\delta_k(t_1,\dots,t_{2k}) =  \prod_{i=1}^k \delta(t_{2i-1} - t_{2i}).
	\end{aligned}
\end{equation*}

\begin{theorem}\label{tm4.2}
In the one-dimensional case, let $\mathcal{B}(0,\varepsilon) = [-\varepsilon, \varepsilon]$ denote the close ball of radius $\varepsilon$ centered at $0$, with boundary $\partial\mathcal{B}(0,\varepsilon) = \{-\varepsilon, \varepsilon\}$.
For any $n\geq1$, we have
\begin{equation*}
	\begin{aligned}
		a^{\varepsilon}_n(t_1,\dots,t_n)=\int_{\partial\mathcal{B}(0,\varepsilon)^n} G_{\varepsilon}(0,t_1,0,z_1) \nabla_{z_{1}}G_{\varepsilon}(t_{1},t_{2},z_{1},z_{2}) \cdots \nabla_{z_{n}}a^{\varepsilon}_{z_{n},t_{n}} dz_1\cdots dz_n,
	\end{aligned}
\end{equation*}
Then the following asymptotic results hold as $\varepsilon\to0$. For any $n=2k+1$ with $k\geq0$,
\begin{equation*}
	\begin{aligned}
		a^{\varepsilon}_{2k+1}(t_1,\dots,t_{2k+1}) = 0.
	\end{aligned}
\end{equation*}
For any $n=2k$ with $k\geq0$, as $\varepsilon\to0$,
\begin{equation*}
	\begin{aligned}
		a^{\varepsilon}_{2k}(t_1,\dots,t_{2k}) \sim \frac{4}{\pi}e^{\frac{\pi^2}{8\varepsilon^{2}}} \cdot \frac{(-1)^{k+1}}{2^k } \prod_{i=1}^k \delta(t_{2i-1} - t_{2i}), 
	\end{aligned}
\end{equation*}
where $0\leq t_{1}\leq\cdots\leq t_{2k}\leq1$ and $\delta(\cdot)$ is the one-dimensional surface delta function. This asymptotic equivalence holds in the sense of weak convergence of multilinear forms in $\mathcal{D}([0,1]^{2k}) = \left\{ \phi:[0,1]^{2k}\to\mathbb{R} \mid \phi \text{ is infinitely differentiable on } [0,1]^{2k} \right\}$.
\end{theorem}

\begin{proof}

By the given conditions, the Fourier--Wiener transform of the hitting indicator $W^{\varepsilon}_{0,0} = \mathbf{1}_{\{\sup_{s\in[0,1]}|w(s)|=\varepsilon\}}$ admits the It\^{o}--Wiener expansion
\begin{equation*}
	\begin{aligned}
		\mathcal{T}(W^{\varepsilon}_{0,0})(h) = \sum_{n=0}^\infty \mathcal{T}_n(h),
	\end{aligned}
\end{equation*}
where the $n$-th order homogeneous term $\mathcal{T}_n(h)$ is given by
\begin{equation}\label{eq:tn_kernel}
	\begin{aligned}
		\mathcal{T}_n(h) = \int_{0\leq t_{1}\leq\cdots\leq t_{n}\leq1} a^{\varepsilon}_n(t_1,\dots,t_n) h(t_1)\cdots h(t_n) dt_1\cdots dt_n.
	\end{aligned}
\end{equation}
By the result in Theorem \ref{tm2.3}, for $n=2k$,
\begin{equation}\label{eq:even_asymp}
	\begin{aligned}
		\lim_{\varepsilon\to0} \mathcal{T}_{2k}(h) e^{\frac{\pi^2}{8}\varepsilon^{-2}} = \frac{4}{\pi} \cdot \frac{(-1)^{k+1}}{2^k k!} \left( \int_0^1 h^2(t) dt \right)^k \quad \forall h\in L^{2}([0,1]).
	\end{aligned}
\end{equation}
For $n=2k+1$,
\begin{equation}\label{eq:odd_asymp}
	\begin{aligned}
		\mathcal{T}_{2k+1}(h)= 0 \quad \forall h\in L^{2}([0,1]).
	\end{aligned}
\end{equation}
We first rewrite the right-hand side of \eqref{eq:even_asymp} as an $n$-fold multiple integral, to match the form of $\mathcal{T}_n(h)$ in \eqref{eq:tn_kernel}. For any $k\geq0$, the $k$-th power of the $L^2$ norm admits the expansion
\begin{equation*}
	\begin{aligned}
		\left( \int_0^1 h^2(t) dt \right)^k =& \int_{[0,1]^{k}} h^2(s_1) h^2(s_2) \cdots h^2(s_k) ds_1 ds_2 \cdots ds_k\\
		=& k! \int_{0\leq s_1\leq\cdots\leq s_k\leq1}h^2(s_1) h^2(s_2) \cdots h^2(s_k) ds_1 ds_2 \cdots ds_k.
	\end{aligned}
\end{equation*}
Substituting \eqref{delta} into \eqref{eq:even_asymp}, we obtain the equivalent asymptotic form
\begin{equation}\label{eq:even_asymp_multivariate}
	\begin{aligned}
		&\lim_{\varepsilon\to0} \mathcal{T}_{2k}(h) e^{\frac{\pi^2}{8\varepsilon^{2}}}\\
		=&\int_{0\leq s_1\leq\cdots\leq s_{2k}\leq1} \left( \frac{4}{\pi} \cdot \frac{(-1)^{k+1}}{2^k } \delta_k(t_1,\dots,t_{2k}) \right) h(t_1)\cdots h(t_{2k}) dt_1\cdots dt_{2k}.
	\end{aligned}
\end{equation}
By \eqref{eq:odd_asymp}, we have
\begin{equation*}
	\begin{aligned}
		a^{\varepsilon}_{2k+1}(t_1,\dots,t_{2k+1}) = 0,
	\end{aligned}
\end{equation*}
which is exactly the first result of the theorem. Observe that $\mathcal{T}_{2k}(h)$ is precisely the restriction to the diagonal of the multilinear form
\begin{equation*}
	\begin{aligned}
		\mathcal{M}_\varepsilon[f_1,\dots,f_{2k}] 
		:= \int_{\Delta} a_{2k}^\varepsilon(t_1,\dots,t_{2k})\, f_1(t_1)\cdots f_{2k}(t_{2k})\,dt_1\cdots dt_{2k},
	\end{aligned}
\end{equation*}
where $\Delta := \{ 0\leq t_1\leq\cdots\leq t_{2k}\leq1 \}$, that is, $\mathcal{T}_{2k}(h) = \mathcal{M}_\varepsilon[h,\dots,h]$. Likewise, the right-hand side of \eqref{eq:even_asymp_multivariate} may be interpreted as a special case of the multilinear form
\begin{equation*}
	\begin{aligned}
		\mathcal{M}_0[f_1,\dots,f_{2k}] 
		:= \frac{4}{\pi}\frac{(-1)^{k+1}}{2^k} 
		\int_{\Delta} \delta_k(t_1,\dots,t_{2k})\, f_1(t_1)\cdots f_{2k}(t_{2k})\,dt_1\cdots dt_{2k}.
	\end{aligned}
\end{equation*}
Equation \eqref{eq:even_asymp_multivariate} shows that for all test tuples in the diagonal subset $\{f_i \equiv h\}$,
\begin{equation}\label{eq:A6}
	\begin{aligned}
		\lim_{\varepsilon\to0} \mathcal{M}_\varepsilon[h,\dots,h] = \mathcal{M}_0[h,\dots,h]. 
	\end{aligned}
\end{equation}
We now use the polarization identity to lift this diagonal convergence to full $2k$-fold multilinear convergence. Since both $\mathcal{M}_\varepsilon$ and $\mathcal{M}_0$ extend to symmetric $2k$-fold multilinear functionals, the polarization identity ensures that the value of any such multilinear form at $(f_1,\dots,f_{2k})$ can be expressed in terms of its diagonal evaluations
\begin{equation*}
	\mathcal{M}[f_1,\dots,f_{2k}] = \frac{1}{(2k)!} \,
	\frac{\partial^{2k}}{\partial\lambda_1\cdots\partial\lambda_{2k}}
	\mathcal{M}\left[\sum_{i=1}^{2k} \lambda_i f_i,\dots,\sum_{i=1}^{2k} \lambda_i f_i\right] \bigg|_{\lambda_1 = \cdots = \lambda_{2k} = 0}.
\end{equation*}
Substituting $h = \sum_{i=1}^{2k} \lambda_i f_i$ into \eqref{eq:A6}, and noting that the limit as $\varepsilon \to 0$ commutes with finite-dimensional linear combinations and partial differentiation, we immediately obtain that for all smooth functions $f_1,\dots,f_{2k} \in C^\infty([0,1])$,
\begin{equation}\label{eq:A7}
	\begin{aligned}
		\lim_{\varepsilon\to0} \mathcal{M}_\varepsilon[f_1,\dots,f_{2k}] = \mathcal{M}_0[f_1,\dots,f_{2k}]. 
	\end{aligned}
\end{equation}
By the Schwartz kernel theorem \cite{Hor90}, every continuous $2k$-fold multilinear functional on $C^\infty([0,1])^{2k}$ corresponds to a unique distribution kernel $K \in \mathcal{D}'([0,1]^{2k})$ such that
\begin{equation*}
	\begin{aligned}
		\mathcal{M}[f_1,\dots,f_{2k}] = \langle K, f_1\otimes\cdots\otimes f_{2k}\rangle.
	\end{aligned}
\end{equation*}
Let us define
\begin{equation*}
	\begin{aligned}
K_\varepsilon := a_{2k}^\varepsilon\, e^{\frac{\pi^2}{8\varepsilon^2}}, \qquad
K_0 := \frac{4}{\pi}\frac{(-1)^{k+1}}{2^k}\,\delta_k,
	\end{aligned}
\end{equation*}
then relation \eqref{eq:A7} implies that for every tensor-product test function $\phi = f_1 \otimes \cdots \otimes f_{2k}$,
\begin{equation*}
	\begin{aligned}
		\lim_{\varepsilon\to0} \langle K_\varepsilon, \phi\rangle = \langle K_0, \phi\rangle.
	\end{aligned}
\end{equation*}
Since finite linear combinations of such tensor-product test functions are dense in $\mathcal{D}([0,1]^{2k})$, and $\langle K_\varepsilon, \phi\rangle$ converges to $\langle K_0, \phi\rangle$ for any $\phi\in\mathcal{D}([0,1]^{2k})$, it follows that the family $\{K_\varepsilon\}$ is uniformly bounded in $\mathcal{D}'([0,1]^{2k})$. By the definition of convergence in the space of distributions, we have
\begin{equation*}
	\begin{aligned}
		K_\varepsilon \xrightarrow{\mathcal{D}'} K_0 \quad (\varepsilon\to 0).
	\end{aligned}
\end{equation*}
Reformulating this distributional convergence in the framework of asymptotic equivalence completes the proof.
\end{proof}

\begin{corollary}
In the $d$-dimensional case,
\begin{equation*}
\begin{aligned}
	\mathcal{T}(W^{\varepsilon}_{0,0})(h)=&a^{\varepsilon}_{0,0}+\sum_{n=1}^{\infty}\int_{0\leq r_{1}\leq\cdots\leq r_{n}\leq1}\int_{\partial {B}(0,\varepsilon)}\stackrel{n}{\cdots}\int_{\partial {B}(0,\varepsilon)}G_{\varepsilon}(0,r_{1},0,z_{1})\\
	&(\nabla_{z_{1}}G_{\varepsilon}(r_{1},r_{2},z_{1},z_{2})\otimes\cdots\otimes\nabla_{z_{n}}a^{\varepsilon}_{z_{n},r_{n}},h(r_{1})\otimes\cdots\otimes h(r_{n}))_{n}\\
	&dz_{1}\cdots dz_{n}dr_{1}\cdots dr_{n},
\end{aligned}
\end{equation*}
Define an operator $B: \R^{d \times 2k} \rightarrow \R $ by
\begin{equation*}
	\begin{aligned}
		B(h(r_{1})\otimes\cdots\otimes h(r_{2k}))=\prod_{j=1}^{k} tr(h(r_{2j-1})\otimes h(r_{2j}))=\prod_{j=1}^{k}(h(r_{2j-1}), h(r_{2j})).
	\end{aligned}
\end{equation*}
From this, one can get
\begin{equation*}
	\begin{aligned}
		&\int_{0\leq t_1\leq\cdots\leq t_{2k}\leq1} \delta_k(t_1,\dots,t_{2k}) A(h(r_{1})\otimes\cdots\otimes h(r_{2k})) dt_1\cdots dt_{2k}\\
		=&\int_{0\leq s_1\leq\cdots\leq s_k\leq1}h^2(s_1) h^2(s_2) \cdots h^2(s_k) ds_1 ds_2 \cdots ds_k\\
		=& \frac{1}{k!} \left( \int_0^1 |h(t)|^{2} dt \right)^k
	\end{aligned}
\end{equation*}
Comparing with the results in Theorem \ref{tm4.2}, if we denote
\begin{equation*}
	\begin{aligned}
		a_{d,2k}^{\varepsilon}=\int_{\partial\mathcal{B}(0,\varepsilon)^n} G_{\varepsilon}(0,t_1,0,z_1) \nabla_{z_{1}}G_{\varepsilon}(t_{1},t_{2},z_{1},z_{2}) \cdots \nabla_{z_{n}}a^{\varepsilon}_{z_{n},t_{n}} dz_1\cdots dz_n,
	\end{aligned}
\end{equation*}
one can get the following results
\begin{equation*}
	\begin{aligned}
		a_{d,2k}^{\varepsilon}\sim \frac{j_{\nu,1}^{\nu-1}(-1)^{k+1}}{2^k\Gamma(\nu+1)2^{\nu-1}J_{\nu+1}(j_{\nu,1})}e^{-j^{2}_{\nu,1}/2\varepsilon^{2}}\, \prod_{i=1}^k \delta(t_{2i-1} - t_{2i}) \quad \varepsilon \to 0.
	\end{aligned}
\end{equation*}
\end{corollary}

\section{Small  deviations for integrators}
Since we have established the small deviation result for the Brownian motion with a shift, we aim to adopt the same approach to derive the corresponding small deviation result for the integrator.
In the one-dimensional case, consider the linear operator $A$ in $L^{2}([0,1])$ of the following form
\begin{equation*}
	\begin{aligned}
		(Ah)(t)=h(t)+\int_{0}^{1}q(t,s)h(s)ds,
	\end{aligned}
\end{equation*}
where $q$ is a continuous kernel on $[0,1]^{2}$ satisfying the inequality
\begin{equation*}
	\begin{aligned}
		\max_{[0,1]^{2}}|q(t,s)|<1.
	\end{aligned}
\end{equation*}
Define integrator
\begin{equation}\label{integratordef}
	\begin{aligned}
		\eta(t)=\int_{0}^{1}A(\mathbf{1}_{[0,t]})(r)dw(r).
	\end{aligned}
\end{equation}

\begin{theorem} \label{tm5.1}
	Consider the integrator $\eta$ defined in (\ref{integratordef}), we can get the small deviation for the integrator as follows
	\begin{equation*}
		\begin{aligned}
			P(\|\eta\|_{\infty}<\varepsilon) \sim \frac{1}{\sqrt{\det(I+Q+Q^*+Q^*Q)}}\frac{4}{\pi}e^{-\frac{\pi^{2}}{8\varepsilon^{2}}}, \quad \varepsilon\rightarrow 0,
		\end{aligned}
	\end{equation*}
	where $Q$ is the integral operator with kernel $q(t,s)$, 
	$Q^*$ is the adjoint of $Q$ with kernel $q(s,t)$, 
	$\det$ is the Fredholm determinant \cite{Sim05} on $H = L^2([0,1])$.
\end{theorem}

\begin{proof}
	
Let
\begin{equation*}
	\begin{aligned}
		M=\max_{[0,1]^2}|q(t,s)|<1.
	\end{aligned}
\end{equation*}
For every $h\in H$, the Cauchy--Schwarz inequality gives
\begin{equation*}
	\begin{aligned}
		\|Qh\|_{H}^{2}
		&=
		\int_{0}^{1}
		\left|
		\int_{0}^{1}q(t,s)h(s)\,ds
		\right|^{2}dt
		\\
		&\leq
		\int_{0}^{1}
		\left(
		\int_{0}^{1}|q(t,s)|^{2}\,ds
		\right)
		\left(
		\int_{0}^{1}|h(s)|^{2}\,ds
		\right)dt
		\\
		&\leq
		M^{2}\|h\|_{H}^{2}.
	\end{aligned}
\end{equation*}
Consequently,
\begin{equation*}
	\begin{aligned}
		\|Q\|\leq M<1.
	\end{aligned}
\end{equation*}
Hence $A=I+Q$ is invertible, with inverse given by the
Neumann series.

Let $W$ denote the isonormal Gaussian process over $H$, so that
\begin{equation*}
	\begin{aligned}
		w(t)=W(\mathbf{1}_{[0,t]}).
	\end{aligned}
\end{equation*}
Then
\begin{equation*}
	\begin{aligned}
		\eta(t)
		=
		W\bigl(A\mathbf{1}_{[0,t]}\bigr).
	\end{aligned}
\end{equation*}
In particular, $\eta$ is a centered Gaussian process with covariance
function
\begin{equation*}
	\begin{aligned}
		K_{\eta}(t,s)
		&=
		\mathbb{E}\bigl[\eta(t)\eta(s)\bigr]
		\\
		&=
		\left\langle
		A\mathbf{1}_{[0,t]},
		A\mathbf{1}_{[0,s]}
		\right\rangle_{H}
		\\
		&=
		\int_{0}^{1}
		A(\mathbf{1}_{[0,t]})(u)
		A(\mathbf{1}_{[0,s]})(u)\,du.
	\end{aligned}
\end{equation*}
Moreover, for $t>0$,
\begin{equation*}
	\begin{aligned}
		\Sigma(t)
		=
		\operatorname{Var}(\eta(t))
		=
		\left\|
		A\mathbf{1}_{[0,t]}
		\right\|_{H}^{2}
		>0,
	\end{aligned}
\end{equation*}
because $A$ is invertible and
$\mathbf{1}_{[0,t]}\neq 0$.

We also have the representation
\begin{equation*}
	\begin{aligned}
		\eta(t)
		&=
		w(t)
		+
		\int_{0}^{1}
		\left(
		\int_{0}^{t}q(r,s)\,ds
		\right)dw(r)
		\\
		&=
		w(t)
		+
		\int_{0}^{t}
		\left(
		\int_{0}^{1}q(r,s)\,dw(r)
		\right)ds,
	\end{aligned}
\end{equation*}
where the second equality follows from the stochastic Fubini
theorem. Thus $\eta$ has a continuous modification.

The Cameron--Martin space of the Wiener measure is
\begin{equation*}
	\begin{aligned}
		H_w
		=
		\left\{
		h\in C([0,1]):
		h(0)=0,\;
		h\text{ is absolutely continuous},\;
		h'\in L^{2}([0,1])
		\right\},
	\end{aligned}
\end{equation*}
equipped with the norm
\begin{equation*}
	\begin{aligned}
		\|h\|_{H_w}^{2}
		=
		\int_{0}^{1}|h'(s)|^{2}\,ds.
	\end{aligned}
\end{equation*}
The derivative map
\begin{equation*}
	\begin{aligned}
		D:H_w\longrightarrow H,
		\qquad Dh=h',
	\end{aligned}
\end{equation*}
is a unitary isomorphism.

Define the Cameron--Martin transfer operator
$U:H_w\to H_w$ by
\begin{equation*}
	\begin{aligned}
		(Uh)(t)
		=
		\int_{0}^{1}
		A(\mathbf{1}_{[0,t]})(r)h'(r)\,dr.
	\end{aligned}
\end{equation*}
A direct calculation gives
\begin{equation*}
	\begin{aligned}
		(Uh)(t)
		&=
		h(t)
		+
		\int_{0}^{1}
		\left( \int_{0}^{t}q(r,s)\,ds \right) h'(r)\,dr\\
		&=
		h(t)
		+
		\int_{0}^{t}
		\int_{0}^{1}
		q(r,s)h'(r)\,dr\,ds.
	\end{aligned}
\end{equation*}
Therefore,
\begin{equation*}
	\begin{aligned}
		(Uh)'(t)
		&=
		h'(t)
		+
		\int_{0}^{1}q(r,t)h'(r)\,dr\\
		&=
		(I+Q^*)h'(t).
	\end{aligned}
\end{equation*}
Thus, under the derivative isometry,
\begin{equation*}
	\begin{aligned}
		DUD^{-1}=I+Q^*.
	\end{aligned}
\end{equation*}

The relative covariance operator of $\mu_\eta$ with respect to
Wiener measure $\mu_w$, represented on $H=L^2([0,1])$, is the
positive self-adjoint operator $S$ determined by
\begin{equation*}
	\begin{aligned}
		R_\eta=JSJ^*,
		\qquad
		R_w=JJ^*,
	\end{aligned}
\end{equation*}
where $(Jf)(t)=\int_0^t f(s)\,ds$. In the present setting, 
\begin{equation*}
	\begin{aligned}
		S
		&=
		(I+Q^*)(I+Q)
		\\
		&=
		I+Q+Q^*+Q^*Q.
	\end{aligned}
\end{equation*}
Since $Q$ is trace class, $S-I$ is trace class. Moreover, $S$ is
positive and invertible. Hence $\mu_{\eta}$ and $\mu_w$ are
equivalent centered Gaussian measures.

Let
\begin{equation*}
	\begin{aligned}
		B_{\varepsilon}
		=
		\left\{
		x\in C([0,1]):
		x(0)=0,\;
		\|x\|_{\infty}<\varepsilon
		\right\}.
	\end{aligned}
\end{equation*}
Let $S=I+Q+Q^*+Q^*Q$, and let $\{s_n\}_{n\geq 1}$ be the
eigenvalues of $S$, counted with multiplicity. Since $S-I$ is
trace class, we have
\begin{equation*}
	\begin{aligned}
		\sum_{n=1}^{\infty}|s_n-1|<\infty.
	\end{aligned}
\end{equation*}
The logarithm of the Radon--Nikodym density of $\mu_\eta$ with
respect to $\mu_w$ admits the representation
\begin{equation*}
	\begin{aligned}
		\log\frac{d\mu_\eta}{d\mu_w}
		=
		c+
		\sum_{n=1}^{\infty}
		a_n\bigl(\xi_n^2-1\bigr),
	\end{aligned}
\end{equation*}
where
\begin{equation*}
	\begin{aligned}
		a_n
		=
		\frac{1}{2}
		\left(1-\frac{1}{s_n}\right),
		\qquad
		c
		=
		-\frac{1}{2}
		\sum_{n=1}^{\infty}\log s_n
		+
		\sum_{n=1}^{\infty}a_n.
	\end{aligned}
\end{equation*}
By the small-ball comparison theorem
\cite[Theorem~4.8.3, pp.~189--191]{Vla15},
\begin{equation*}
	\begin{aligned}
		\lim_{\varepsilon\downarrow0}
		\frac{\mu_\eta(B_\varepsilon)}
		{\mu_w(B_\varepsilon)}
		&=
		\exp\left(
		c-\sum_{n=1}^{\infty}a_n
		\right)
		\\
		&=
		\exp\left(
		-\frac{1}{2}
		\sum_{n=1}^{\infty}\log s_n
		\right)
		\\
		&=
		\frac{1}{\sqrt{\det S}}.
	\end{aligned}
\end{equation*}

Finally, combine with the small deviation for the Brownian motion and using
\begin{equation*}
	\begin{aligned}
		S=I+Q+Q^*+Q^*Q
	\end{aligned}
\end{equation*}
gives
\begin{equation*}
	\begin{aligned}
		P\bigl(\|\eta\|_{\infty}<\varepsilon\bigr)
		\sim
		\frac{1}
		{\sqrt{\det(I+Q+Q^*+Q^*Q)}}
		\frac{4}{\pi}
		\exp\left(
		-\frac{\pi^2}{8\varepsilon^2}
		\right),
		\qquad \varepsilon\rightarrow 0.
	\end{aligned}
\end{equation*}
\end{proof}

Next, we investigate the small deviation of a more complicated process $X(t)=\eta(t)+\int_{0}^{t}h(s)\,ds$ .

\begin{theorem}
		Under the assumptions of Theorem~\ref{tm5.1}, assume that $Q$ is a trace-class operator on $L^2([0,1])$. For $h\in L^2([0,1])$, consider the process
	\begin{equation*}
		\begin{aligned}
			X(t)
			=
			\eta(t)+\int_{0}^{t}h(s)\,ds,
			\qquad t\in[0,1].
		\end{aligned}
	\end{equation*}
	Then the small-deviation probability of $X$ satisfies
	\begin{equation*}
		\begin{aligned}
			P\left(\|X\|_{\infty}<\varepsilon\right)
			\sim
			\frac{1}
			{\sqrt{\det(I+Q+Q^*+Q^*Q)}}
			\frac{4}{\pi}
			\exp\left(
			-\frac{\pi^{2}}{8\varepsilon^{2}}
			-\frac{1}{2}
			\left\|(I+Q^*)^{-1}h\right\|_{L^2([0,1])}^{2}
			\right),
			\qquad \varepsilon\to0,
		\end{aligned}
	\end{equation*}
	where $Q^*$ is the adjoint of $Q$, whose kernel is $q(s,t)$, and
	$\det$ denotes the Fredholm determinant on
	$H=L^2([0,1])$.
\end{theorem}

\begin{proof}
	By Theorem~\ref{tm5.1}, $A=I+Q$ is invertible. Therefore, its
	adjoint
	\begin{equation*}
		\begin{aligned}
			A^*=I+Q^*
		\end{aligned}
	\end{equation*}
	is also invertible.
	
	Let $H=L^2([0,1])$. The Cameron--Martin space $H_\eta$ of the
	Gaussian process $\eta$ consists of all functions of the form
	\begin{equation*}
		\begin{aligned}
			f_g(t)
			=
			\int_0^1
			A(\mathbf{1}_{[0,t]})(r)g(r)\,dr,
			\qquad g\in H.
		\end{aligned}
	\end{equation*}
	Equivalently,
	\begin{equation*}
		\begin{aligned}
			f_g(t)
			&=
			\left\langle
			A\mathbf{1}_{[0,t]},g
			\right\rangle_H
			\\
			&=
			\left\langle
			\mathbf{1}_{[0,t]},A^*g
			\right\rangle_H
			\\
			&=
			\int_0^t(A^*g)(s)\,ds.
		\end{aligned}
	\end{equation*}
	Since $A$ is invertible and the linear span of
	$\{\mathbf{1}_{[0,t]}:0\leq t\leq1\}$ is dense in $H$, the
	representation $f_g$ is unique. Consequently, the norm in
	$H_\eta$ is given by
	\begin{equation*}
		\begin{aligned}
			\|f_g\|_{H_\eta}
			=
			\|g\|_H.
		\end{aligned}
	\end{equation*}
	
	Set
	\begin{equation*}
		\begin{aligned}
			m(t)
			=
			\int_0^t h(s)\,ds.
		\end{aligned}
	\end{equation*}
	The condition $m=f_g$ is equivalent to
	\begin{equation*}
		\begin{aligned}
			h=m'=A^*g.
		\end{aligned}
	\end{equation*}
	Hence
	\begin{equation*}
		\begin{aligned}
			g
			=
			(A^*)^{-1}h
			=
			(I+Q^*)^{-1}h.
		\end{aligned}
	\end{equation*}
	Thus $m\in H_\eta$ for every $h\in L^2([0,1])$, and
	\begin{equation*}
		\begin{aligned}
			\|m\|_{H_\eta}^{2}
			&=
			\left\|(A^*)^{-1}h\right\|_{L^2([0,1])}^{2}
			\\
			&=
			\left\langle
			(A^*A)^{-1}h,h
			\right\rangle_{L^2([0,1])}.
		\end{aligned}
	\end{equation*}
	
	Let $\mu_\eta$ and $\mu_X$ denote the laws of $\eta$ and $X$,
	respectively, on $C([0,1])$. Since $m\in H_\eta$, the
	Cameron--Martin theorem implies that $\mu_X$ and $\mu_\eta$ are
	equivalent. More precisely,
	\begin{equation*}
		\begin{aligned}
			\frac{d\mu_X}{d\mu_\eta}(\omega)
			=
			\exp\left(
			\widehat{m}(\omega)
			-\frac{1}{2}\|m\|_{H_\eta}^{2}
			\right),
		\end{aligned}
	\end{equation*}
	where $\widehat{m}$ is the measurable linear functional
	associated with $m\in H_\eta$. Under the representation of
	$\eta$ in terms of the Brownian motion $w$, it satisfies
	\begin{equation*}
		\begin{aligned}
			\widehat{m}(\eta)
			&=
			\int_0^1g(r)\,dw(r)
			\\
			&=
			\int_0^1
			\left[(A^*)^{-1}h\right](r)\,dw(r).
		\end{aligned}
	\end{equation*}
	Notice that $\widehat{m}(\omega)$ is, in general, not a
	continuous linear functional on $C([0,1])$ equipped with the
	supremum norm.
	
	For $a\in C([0,1])$, define
	\begin{equation*}
		\begin{aligned}
			B(a,\varepsilon)
			=
			\left\{
			\omega\in C([0,1]):
			\|\omega-a\|_\infty<\varepsilon
			\right\}.
		\end{aligned}
	\end{equation*}
	Since $X=\eta+m$, we have
	\begin{equation*}
		\begin{aligned}
			P\left(\|X\|_\infty<\varepsilon\right)
			&=
			P\left(\|\eta+m\|_\infty<\varepsilon\right)
			\\
			&=
			\mu_\eta\bigl(B(-m,\varepsilon)\bigr).
		\end{aligned}
	\end{equation*}
	
	The shifted small-ball comparison formula for a centered Gaussian
	measure states that, for every $m\in H_\eta$,
	\begin{equation*}
		\begin{aligned}
			\lim_{\varepsilon\downarrow0}
			\frac{
				\mu_\eta(B(-m,\varepsilon))
			}{
				\mu_\eta(B(0,\varepsilon))
			}
			=
			\exp\left(
			-\frac{1}{2}\|m\|_{H_\eta}^{2}
			\right).
		\end{aligned}
	\end{equation*}
	Therefore,
	\begin{equation*}
		\begin{aligned}
			P\left(\|X\|_\infty<\varepsilon\right)
			\sim
			\exp\left(
			-\frac{1}{2}
			\left\|(I+Q^*)^{-1}h\right\|_{L^2([0,1])}^{2}
			\right)
			P\left(\|\eta\|_\infty<\varepsilon\right).
		\end{aligned}
	\end{equation*}
	
	Finally, Theorem~\ref{tm5.1} gives
	\begin{equation*}
		\begin{aligned}
			P\left(\|\eta\|_\infty<\varepsilon\right)
			\sim
			\frac{1}
			{\sqrt{\det(I+Q+Q^*+Q^*Q)}}
			\frac{4}{\pi}
			\exp\left(
			-\frac{\pi^2}{8\varepsilon^2}
			\right).
		\end{aligned}
	\end{equation*}
	Combining the last two asymptotic relations, we obtain
	\begin{equation*}
		\begin{aligned}
			P\left(\|X\|_\infty<\varepsilon\right)
			\sim
			\frac{1}
			{\sqrt{\det(I+Q+Q^*+Q^*Q)}}
			\frac{4}{\pi}
			\exp\left(
			-\frac{\pi^2}{8\varepsilon^2}
			-\frac{1}{2}
			\left\|(I+Q^*)^{-1}h\right\|_{L^2([0,1])}^{2}
			\right),
			\qquad \varepsilon\to0.
		\end{aligned}
	\end{equation*}
\end{proof}

For $x\in \mathcal{D}$ (where $\mathcal{D} \subset \mathbb{R}^{d}$ is a domain with smooth boundary) and $s\in[0,1]$, define the indicator random variable
\begin{equation*}
	W_{\eta,x,s} := 
	\begin{cases}
		1, & \text{if there exists } \tau\in[s,1] \text{ such that } x+\eta(\tau)-\eta(s)\in\partial \mathcal{D}, \\
		0, & \text{otherwise}.
	\end{cases}
\end{equation*}
The expectation $\mathbb{E}[W_{\eta,x,s}]$ is the probability that the integrator $\eta$ starting from $x$ at time $s$ hits the boundary $\partial \mathcal{D}$ by time $1$. The following theorem holds.

\begin{theorem}\label{tm5.3}
	Under the assumptions of Theorem~\ref{tm5.1}, consider the
	one-dimensional case and let $h\in L^2([0,1])$. Then
	\begin{equation*}
		\begin{aligned}
			1-\mathcal{T}
			(W^\varepsilon_{\eta,0,0})(h)
			\sim
			\frac{4}
			{\pi\sqrt{\det(I+Q+Q^*+Q^*Q)}}
			\exp\left(
			-\frac{1}{2}
			\int_0^1|h(s)|^2\,ds
			-\frac{\pi^2}{8\varepsilon^2}
			\right),
			\qquad \varepsilon\to0,
		\end{aligned}
	\end{equation*}
	where $W^\varepsilon_{\eta,0,0}$ corresponds to the exit event
	from the interval $\mathcal D_\varepsilon=(-\varepsilon,\varepsilon)$,
	and $\det$ denotes the Fredholm determinant on $L^2([0,1])$.
\end{theorem}
\begin{proof}
	Let
	\begin{equation*}
		\begin{aligned}
			H=L^2([0,1])
		\end{aligned}
	\end{equation*}
	and set
	\begin{equation*}
		\begin{aligned}
			c=\max_{(t,s)\in[0,1]^2}|q(t,s)|<1.
		\end{aligned}
	\end{equation*}
	For every $f\in H$, the Cauchy--Schwarz inequality gives
	\begin{equation*}
		\begin{aligned}
			\|Qf\|_H^2
			&=
			\int_0^1
			\left|
			\int_0^1q(t,s)f(s)\,ds
			\right|^2dt
			\\
			&\leq
			\int_0^1\int_0^1|q(t,s)|^2\,ds\,dt
			\|f\|_H^2
			\\
			&\leq
			c^2\|f\|_H^2.
		\end{aligned}
	\end{equation*}
	Thus $\|Q\|\leq c<1$, and consequently
	\begin{equation*}
		\begin{aligned}
			A^{-1}
			=
			(I+Q)^{-1}
			=
			\sum_{n=0}^{\infty}(-Q)^n
		\end{aligned}
	\end{equation*}
	converges in the operator norm. Hence $A=I+Q$ and
	\begin{equation*}
		\begin{aligned}
			A^*=I+Q^*
		\end{aligned}
	\end{equation*}
	are boundedly invertible.
	
	Define the stochastic exponential
	\begin{equation*}
		\begin{aligned}
			\mathcal E_h(w)
			=
			\exp\left(
			\int_0^1h(s)\,dw(s)
			-\frac{1}{2}\|h\|_H^2
			\right).
		\end{aligned}
	\end{equation*}
	By the definition of the Fourier--Wiener transform,
	\begin{equation*}
		\begin{aligned}
			\mathcal T(F)(h)
			=
			\mathbb E\left[
			F\mathcal E_h(w)
			\right].
		\end{aligned}
	\end{equation*}
	Moreover,
	\begin{equation*}
		\begin{aligned}
			\mathbb E[\mathcal E_h(w)]
			=1.
		\end{aligned}
	\end{equation*}
	It follows that
	\begin{equation*}
		\begin{aligned}
			1-\mathcal T
			(W^\varepsilon_{\eta,0,0})(h)
			&=
			\mathbb E\left[
			\left(
			1-W^\varepsilon_{\eta,0,0}
			\right)
			\mathcal E_h(w)
			\right].
		\end{aligned}
	\end{equation*}
	
	Because $\eta(0)=0$ and $\eta$ has continuous sample paths, the
	event that $\eta$ does not hit
	\begin{equation*}
		\begin{aligned}
			\partial\mathcal D_\varepsilon
			=
			\{-\varepsilon,\varepsilon\}
		\end{aligned}
	\end{equation*}
	during $[0,1]$ is precisely the event that its entire path remains
	in $\mathcal D_\varepsilon$. Hence
	\begin{equation*}
		\begin{aligned}
			1-W^\varepsilon_{\eta,0,0}
			=
			\mathbf 1_{\{\|\eta\|_\infty<\varepsilon\}}.
		\end{aligned}
	\end{equation*}
	Therefore,
	\begin{equation}\label{FWsmallball}
		\begin{aligned}
			1-\mathcal T
			(W^\varepsilon_{\eta,0,0})(h)
			=
			\mathbb E\left[
			\mathbf 1_{\{\|\eta\|_\infty<\varepsilon\}}
			\mathcal E_h(w)
			\right].
		\end{aligned}
	\end{equation}
	
	We now apply the Cameron--Martin formula at the level of the
	underlying isonormal Gaussian process. Let $W$ be the isonormal
	Gaussian process over $H$, so that
	\begin{equation*}
		\begin{aligned}
			W(f)=\int_0^1f(s)\,dw(s),
			\qquad f\in H,
		\end{aligned}
	\end{equation*}
	and
	\begin{equation*}
		\begin{aligned}
			\eta(t)
			=
			W\left(
			A\mathbf 1_{[0,t]}
			\right).
		\end{aligned}
	\end{equation*}
	For the fixed element $h\in H$, define the shifted isonormal
	process
	\begin{equation*}
		\begin{aligned}
			W_h(f)
			=
			W(f)+\langle f,h\rangle_H,
			\qquad f\in H.
		\end{aligned}
	\end{equation*}
	The Cameron--Martin formula states that, for every bounded
	measurable functional $\Phi$ of $W$,
	\begin{equation}\label{isonormalCM}
		\begin{aligned}
			\mathbb E\left[
			\Phi(W)
			\exp\left(
			W(h)-\frac{1}{2}\|h\|_H^2
			\right)
			\right]
			=
			\mathbb E[\Phi(W_h)].
		\end{aligned}
	\end{equation}
	
	Under the shift from $W$ to $W_h$, for each $t\in[0,1]$,
	the random variable $\eta(t)$ is transformed as follows
	\begin{equation*}
		\begin{aligned}
			W_h\left(
			A\mathbf 1_{[0,t]}
			\right)
			&=
			W\left(
			A\mathbf 1_{[0,t]}
			\right)
			+
			\left\langle
			A\mathbf 1_{[0,t]},h
			\right\rangle_H
			\\
			&=
			\eta(t)+u_h(t),
		\end{aligned}
	\end{equation*}
	where
	\begin{equation}\label{definitionuh}
		\begin{aligned}
			u_h(t)
			&=
			\left\langle
			A\mathbf 1_{[0,t]},h
			\right\rangle_H
			\\
			&=
			\left\langle
			\mathbf 1_{[0,t]},A^*h
			\right\rangle_H
			\\
			&=
			\int_0^t(A^*h)(s)\,ds.
		\end{aligned}
	\end{equation}
	Taking
	\begin{equation*}
		\begin{aligned}
			\Phi(W)
			=
			\mathbf 1_{\{
				\sup_{0\leq t\leq1}
				|W(A\mathbf 1_{[0,t]})|
				<\varepsilon
				\}}
		\end{aligned}
	\end{equation*}
	in \eqref{isonormalCM}, we obtain the exact identity
	\begin{equation}\label{shiftedball}
		\begin{aligned}
			1-\mathcal T
			(W^\varepsilon_{\eta,0,0})(h)
			=
			P\left(
			\|\eta+u_h\|_\infty<\varepsilon
			\right).
		\end{aligned}
	\end{equation}
	
	It remains to compute the Cameron--Martin norm of $u_h$. Define
	the linear operator
	\begin{equation*}
		\begin{aligned}
			\mathcal K:H\longrightarrow C([0,1])
		\end{aligned}
	\end{equation*}
	by
	\begin{equation*}
		\begin{aligned}
			(\mathcal Kg)(t)
			&=
			\left\langle
			A\mathbf 1_{[0,t]},g
			\right\rangle_H
			\\
			&=
			\int_0^t(A^*g)(s)\,ds.
		\end{aligned}
	\end{equation*}
	The Cauchy--Schwarz inequality shows that
	\begin{equation*}
		\begin{aligned}
			|(\mathcal Kg)(t)-(\mathcal Kg)(s)|
			&=
			\left|
			\int_s^t(A^*g)(r)\,dr
			\right|
			\\
			&\leq
			\|A^*g\|_H|t-s|^{1/2},
		\end{aligned}
	\end{equation*}
	so $\mathcal Kg\in C([0,1])$.
	
	The Cameron--Martin space associated with the Gaussian process
	$\eta(t)=W(A\mathbf 1_{[0,t]})$ is the range of $\mathcal K$,
	equipped with the quotient norm induced by $H$. In the present
	case, $\mathcal K$ is injective. Indeed, if $\mathcal Kg=0$, then
	\begin{equation*}
		\begin{aligned}
			\int_0^t(A^*g)(s)\,ds=0
			\qquad\text{for every }t\in[0,1].
		\end{aligned}
	\end{equation*}
	Differentiating this absolutely continuous function gives
	\begin{equation*}
		\begin{aligned}
			A^*g=0
			\qquad\text{almost everywhere}.
		\end{aligned}
	\end{equation*}
	Since $A^*$ is invertible, it follows that $g=0$. Consequently,
	\begin{equation*}
		\begin{aligned}
			H_\eta
			&=
			\mathcal K(H)
			\\
			&=
			\left\{
			u_g:
			u_g(t)=\int_0^t(A^*g)(s)\,ds,\quad g\in H
			\right\},
		\end{aligned}
	\end{equation*}
	and its inner product is
	\begin{equation*}
		\begin{aligned}
			\langle u_{g_1},u_{g_2}\rangle_{H_\eta}
			=
			\langle g_1,g_2\rangle_H.
		\end{aligned}
	\end{equation*}
	In particular, the function $u_h$ in \eqref{definitionuh}
	corresponds to the element $h\in H$, and hence
	\begin{equation}\label{CMnormh}
		\begin{aligned}
			\|u_h\|_{H_\eta}^2
			=
			\|h\|_H^2
			=
			\int_0^1|h(s)|^2\,ds.
		\end{aligned}
	\end{equation}
	
	Let $\mu_\eta$ denote the law of $\eta$ on $C([0,1])$ and define
	\begin{equation*}
		\begin{aligned}
			B(a,\varepsilon)
			=
			\left\{
			\omega\in C([0,1]):
			\|\omega-a\|_\infty<\varepsilon
			\right\}.
		\end{aligned}
	\end{equation*}
	Then
	\begin{equation*}
		\begin{aligned}
			P\left(
			\|\eta+u_h\|_\infty<\varepsilon
			\right)
			=
			\mu_\eta\left(
			B(-u_h,\varepsilon)
			\right).
		\end{aligned}
	\end{equation*}
	The shifted small-ball comparison theorem for centered Gaussian
	measures states that, for every $a\in H_\eta$,
	\begin{equation*}
		\begin{aligned}
			\lim_{\varepsilon\downarrow0}
			\frac{
				\mu_\eta(B(a,\varepsilon))
			}{
				\mu_\eta(B(0,\varepsilon))
			}
			=
			\exp\left(
			-\frac{1}{2}\|a\|_{H_\eta}^2
			\right).
		\end{aligned}
	\end{equation*}
	Applying this theorem with $a=-u_h$ and using
	\eqref{CMnormh}, we obtain
	\begin{equation}\label{shiftedcomparison}
		\begin{aligned}
			P\left(
			\|\eta+u_h\|_\infty<\varepsilon
			\right)
			&\sim
			\exp\left(
			-\frac{1}{2}\|u_h\|_{H_\eta}^2
			\right)
			P\left(
			\|\eta\|_\infty<\varepsilon
			\right)
			\\
			&=
			\exp\left(
			-\frac{1}{2}
			\int_0^1|h(s)|^2\,ds
			\right)
			P\left(
			\|\eta\|_\infty<\varepsilon
			\right).
		\end{aligned}
	\end{equation}
	
	Finally, Theorem~\ref{tm5.1} gives
	\begin{equation}\label{centeredsmallball}
		\begin{aligned}
			P\left(
			\|\eta\|_\infty<\varepsilon
			\right)
			\sim
			\frac{4}
			{\pi\sqrt{\det(I+Q+Q^*+Q^*Q)}}
			\exp\left(
			-\frac{\pi^2}{8\varepsilon^2}
			\right).
		\end{aligned}
	\end{equation}
	Combining \eqref{shiftedball},
	\eqref{shiftedcomparison}, and
	\eqref{centeredsmallball}, we conclude that
	\begin{equation*}
		\begin{aligned}
			1-\mathcal T
			(W^\varepsilon_{\eta,0,0})(h)
			\sim
			\frac{4}
			{\pi\sqrt{\det(I+Q+Q^*+Q^*Q)}}
			\exp\left(
			-\frac{1}{2}
			\int_0^1|h(s)|^2\,ds
			-\frac{\pi^2}{8\varepsilon^2}
			\right),
			\qquad \varepsilon\to0.
		\end{aligned}
	\end{equation*}
\end{proof}

Based on the It\^{o}--Wiener expansion of
$W^{\varepsilon}_{\eta,0,0}$, we represent its Fourier--Wiener
transform via the series expansion
\begin{equation*}
	\begin{aligned}
		\mathcal{T}(W^{\varepsilon}_{\eta,0,0})(h)
		=
		\sum_{n=0}^{\infty}
		\mathcal{T}^{\varepsilon}_{\eta,n}(h).
	\end{aligned}
\end{equation*}
We can establish the following result.

\begin{theorem}
	Assume that the conditions of Theorem~\ref{tm5.3} hold. Then, for every positive integer $k$,
	\begin{equation*}
		\begin{aligned}
			\lim_{\varepsilon\to0}
			\mathcal{T}^{\varepsilon}_{\eta,2k}(h)
			e^{\pi^2/(8\varepsilon^2)}
			=
			\frac{4}
			{\pi\sqrt{\det(I+Q+Q^*+Q^*Q)}}
			\frac{(-1)^{k+1}}{2^k k!}
			\|h\|^{2k}.
		\end{aligned}
	\end{equation*}
	The zeroth-order coefficient satisfies
	\begin{equation*}
		\begin{aligned}
			\lim_{\varepsilon\to0}
			\left(
			1-\mathcal{T}^{\varepsilon}_{\eta,0}
			\right)
			e^{\pi^2/(8\varepsilon^2)}
			=
			\frac{4}
			{\pi\sqrt{\det(I+Q+Q^*+Q^*Q)}}.
		\end{aligned}
	\end{equation*}
	Moreover, for every nonnegative integer $k$,
	\begin{equation*}
		\begin{aligned}
			\mathcal{T}^{\varepsilon}_{\eta,2k+1}(h)
			=
			0.
		\end{aligned}
	\end{equation*}
\end{theorem}

\begin{proof}
	Define
	\begin{equation*}
		\begin{aligned}
			K
			=
			\frac{4}
			{\pi\sqrt{\det(I+Q+Q^*+Q^*Q)}}
		\end{aligned}
	\end{equation*}
	and
	\begin{equation*}
		\begin{aligned}
			a_{\varepsilon}
			=
			\frac{\pi^2}{8\varepsilon^2}.
		\end{aligned}
	\end{equation*}
	By Theorem~\ref{tm5.3}, for every fixed $h\in L^2([0,1])$,
	\begin{equation*}
		\begin{aligned}
			\mathcal{T}
			(W^{\varepsilon}_{\eta,0,0})(h)
			=
			1
			-
			K
			e^{-\frac{1}{2}\|h\|^2}
			e^{-a_{\varepsilon}}
			+
			o\left(e^{-a_{\varepsilon}}\right).
		\end{aligned}
	\end{equation*}
	In particular, by taking $h=0$, we obtain
	\begin{equation*}
		\begin{aligned}
			\mathcal{T}^{\varepsilon}_{\eta,0}
			=
			\mathcal{T}
			(W^{\varepsilon}_{\eta,0,0})(0)
			=
			1
			-
			K e^{-a_{\varepsilon}}
			+
			o\left(e^{-a_{\varepsilon}}\right).
		\end{aligned}
	\end{equation*}
	Consequently,
	\begin{equation*}
		\begin{aligned}
			\lim_{\varepsilon\to0}
			\left(
			1-\mathcal{T}^{\varepsilon}_{\eta,0}
			\right)
			e^{a_{\varepsilon}}
			=
			K.
		\end{aligned}
	\end{equation*}
	
	For $z$ in a complex neighborhood of zero, the homogeneity of
	the Fourier--Wiener transform of the $n$-th Wiener chaos gives
	\begin{equation*}
		\begin{aligned}
			\mathcal{T}^{\varepsilon}_{\eta,n}(zh)
			=
			z^n
			\mathcal{T}^{\varepsilon}_{\eta,n}(h).
		\end{aligned}
	\end{equation*}
	Hence
	\begin{equation*}
		\begin{aligned}
			\mathcal{T}
			(W^{\varepsilon}_{\eta,0,0})(zh)
			=
			\sum_{n=0}^{\infty}
			z^n
			\mathcal{T}^{\varepsilon}_{\eta,n}(h).
		\end{aligned}
	\end{equation*}
	Applying Theorem~\ref{tm5.3} with $zh$ in place of $h$ yields
	\begin{equation*}
		\begin{aligned}
			\mathcal{T}
			(W^{\varepsilon}_{\eta,0,0})(zh)
			=
			1
			-
			K
			e^{-\frac{1}{2}z^2\|h\|^2}
			e^{-a_{\varepsilon}}
			+
			o\left(e^{-a_{\varepsilon}}\right),
		\end{aligned}
	\end{equation*}
	where the remainder is locally uniform with respect to $z$.
	
	Subtracting the zeroth-order term, we obtain
	\begin{equation}\label{Taylor1}
		\begin{aligned}
			\sum_{n=1}^{\infty}
			z^n
			\mathcal{T}^{\varepsilon}_{\eta,n}(h)
			&=
			K
			\left(
			1-
			e^{-\frac{1}{2}z^2\|h\|^2}
			\right)
			e^{-a_{\varepsilon}}
			+
			o\left(e^{-a_{\varepsilon}}\right)
			\\
			&=
			K
			\sum_{m=1}^{\infty}
			\frac{(-1)^{m+1}}
			{2^m m!}
			z^{2m}\|h\|^{2m}
			e^{-a_{\varepsilon}}
			+
			o\left(e^{-a_{\varepsilon}}\right).
		\end{aligned}
	\end{equation}
	The local uniformity allows the coefficients of the analytic
	power series in \eqref{Taylor1} to be compared. Therefore, for
	every positive integer $k$,
	\begin{equation*}
		\begin{aligned}
			\lim_{\varepsilon\to0}
			\mathcal{T}^{\varepsilon}_{\eta,2k}(h)
			e^{a_{\varepsilon}}
			=
			K
			\frac{(-1)^{k+1}}
			{2^k k!}
			\|h\|^{2k}.
		\end{aligned}
	\end{equation*}
	
	Finally, the functional $W^{\varepsilon}_{\eta,0,0}$ is invariant
	under the transformation $w\mapsto-w$. Indeed, $\eta$ is linear
	in $w$, and the defining small-ball event is invariant under
	$\eta\mapsto-\eta$. On the other hand, the $n$-th Wiener chaos
	changes by the factor $(-1)^n$ under $w\mapsto-w$. It follows
	that all odd-order chaos projections vanish. Hence, for every
	nonnegative integer $k$,
	\begin{equation*}
		\begin{aligned}
			\mathcal{T}^{\varepsilon}_{\eta,2k+1}(h)
			=
			0.
		\end{aligned}
	\end{equation*}
\end{proof}

\end{document}